\documentclass{amsart}

\usepackage[margin=2cm]{geometry}
\usepackage{graphics, amsmath,enumitem, comment, mathtools}
\usepackage{graphicx}
\usepackage{epsfig}

\newif\ifcolorcomments
\newcommand{\allowcomments}[4]{
\newcommand{#1}[1]{\ifdraft{\ifcolorcomments{\textcolor{#4}{##1 --#3}}\else{\textsl{ ##1 \ --#3}}\fi}\else{}\fi}
}

\allowcomments{\comMR}{MR}{MR}{green}

\allowcomments{\comAB}{AB}{AB}{magenta}

\newcommand {\ignore}[1] {}

\colorcommentstrue
\usepackage{amssymb}
\usepackage{amsmath,amsthm}
\usepackage{mathrsfs}
\usepackage{bbm}
\def\bc{\begin{center}}
\def\ec{\end{center}}
\def\be{\begin{equation}}
\def\ee{\end{equation}}

\def\N{\mathbb N}

\def\R{\mathbb R}

\newtheorem{lem}{Lemma}[section]

\newtheorem{pro}[lem]{Proposition}
\newtheorem{thm}[lem]{Theorem}

\newtheorem{rem}[lem]{Remark}
\numberwithin{equation}{section}
\usepackage{colortbl}

\newif\ifdraft\drafttrue

\DeclareMathOperator{\Dist}{Dist}

\newcommand*{\myDots}{\ifmmode\mathellipsis\else.\kern-0.07em.\kern-0.07em.\fi}

\begin{document}

%\subjclass[2010]  {11K50 (11J70, 11J83, 28A78, 28A80)}

\title[ \MakeLowercase{d}-decaying Gauss like systems]{ Hausdorff dimension for the weighted products of multiple digits  in \MakeLowercase{d}-decaying Gauss like systems }

\author[A.~Bakhtawar]{Ayreena~Bakhtawar}
	\address{%
Institute of Mathematics, Polish Academy of Sciences, ul.  \'Sniadeckich 8, 00-656
Warszawa, Poland%
}
	\email{\href{ayreena.bakhtawar@gmail.com}{ayreena.bakhtawar@gmail.com},  \href{abakhtawar@impan.pl}{abakhtawar@impan.pl} }

\author[M.~Rams]{Micha\l~Rams}
	\address{   Institute of Mathematics, Polish Academy of Sciences, ul.  \'Sniadeckich 8, 00-656
Warszawa, Poland         }
	\email{\href{rams@impan.pl}{rams@impan.pl}}

\maketitle

   \renewcommand\thefootnote{}\footnote{The research of A.B. is partially supported by the research project ‘Dynamics and Information Research Institute - Quantum Information, Quantum Technologies’ within the agreement between UniCredit Bank and Scuola Normale Superiore (SNS), Pisa. The first author also acknowledges the support of the Centro di Ricerca Matematica Ennio de Giorgi, SNS and the Institute of Mathematics of the Polish Academy of Sciences (IMPAN)
for providing excellent working conditions and research travel funds. \\    
  The authors thank the ERD\"OS Center Budapest for their hospitality during the Workshop  on Dimension Theory of Fractals 2024, where some of the discussions were initiated.}
  
\begin{abstract}
We compute the Hausdorff dimension of sets defined by the growth of weighted products of multiple digits at arbitrary positions in $d$-decaying Gauss-like iterated function systems.
We provide the complete Hausdorff dimensional result for product of more than two digits, which was an open problem even for consecutive digits in the classical Gauss map and L\"{u}roth map. In our approach we do not need to assume the Bounded Distortion Property (BDP). 
\end{abstract}

%\tableofcontents

\section{Introduction}

The study of infinite iterated function systems (iIFS) on the unit interval represents a rich intersection of ergodic theory, number theory, and fractal geometry. A particularly important class consists of $d$-decaying systems, characterized by a polynomial decay rate in the derivatives of the generating maps. The foundational work in this area was established by Mauldin and Urbański \cite{MU96, MU} on infinite conformal iterated function systems, and later specialized by Jordan and Rams \cite{JR12} in their study of $d$-decaying systems.

This paper has two closely related aims. We develop the distortion
and pressure theory of $d$-decaying Gauss-like systems without assuming
the Bounded Distortion Property. In particular, we establish Almost
Tempered Distortion for the full system, Tempered Distortion on every
fixed finite subsystem, and the existence of geometric pressure defined
by cylinder lengths. We then use these results to determine the
Hausdorff dimension of sets defined by the growth of an arbitrary
finite product of digits raised to arbitrary positive powers, at
arbitrary fixed positions. This resolves the weighted dimension problem
left open beyond two consecutive digits in the continued-fraction and
L\"uroth settings in the regime $1<B<\infty$; see
\cite[Section~8]{BHKW} and \cite[Section~6]{BAH23}.

To treat different expansions in a common framework, we work within the following setting known as  $d$-decaying Gauss like iterated function systems.
Denote by $\N=\{1,2,\cdots,\}$ the set of  positive integers. Let $d>1$ be a real number. A family $\{T_{i}\}_{i\geq1}$  of $C^1$ maps given by  $T_{i}:[0,1]\to [0,1]$ is called $d$-decaying Gauss-like iterated function systems if the following properties are satisfied:

\begin{itemize}
\item[(i)] Open set condition: For any $i\neq j\in \N, $  $T_{i}(0,1)\cap T_{j}(0,1)=\emptyset;$
\item[(ii)] $\bigcup\limits{^{\infty}_{i=1}} T_{i}([0,1])=[0,1);$
\item[(iii)] if $T_{i}(x)<T_{j}(x)$ for all $x\in(0,1)$ then $i<j;$
\item [(iv)] there exists $m\in\N$ and  a real number $0<h<1$ such that for all $(a_{1},\cdots,a_{m})\in\N^{m}$ and $x\in[0,1]$
$$0<\left \vert  (T_{a_{1}}\circ\cdots\circ T_{a_{m}} )'   (x)     \right\vert\leq h <1; $$
\item [(v)] for any $\varepsilon  >0,$ we have constants  $K_{1}=K_{1}(\varepsilon  ),$ $K_{2}=K_{2}(  \varepsilon  )>0,$ such that for each $i\in\N$ there exists $\zeta_{i},\lambda_{i}$ such that for all $x\in[0,1],$
\begin{equation*}
\zeta_{i}\leq \left\vert T'_{i}(x)  \right\vert \leq \lambda_{i} \text{ and } \frac{K_{1}}{i^{d+\varepsilon }}\leq\zeta_{i}\leq\lambda_{i}\leq \frac{K_{2}}{i^{d- \varepsilon  }}.
\end{equation*}
\end{itemize}
The ordering condition means that the first level branch intervals accumulate at a single endpoint. The reverse ordering is equivalent after reflecting the interval. As usual, we use a fixed convention at branch endpoint; an equivalent formulation with the opposite endpoint omitted gives the same dimension theory.

The maps $T_i$ are inverse branches, rather than expanding maps.
Condition~(iv) imposes uniform contraction on a fixed iterate, not
necessarily on each individual branch. We write $T$ for the associated
interval map, defined by $T(x)=T_i^{-1}(x)$ in the interior of
$T_i([0,1]).$
 For a finite word
$\omega=(\omega_1,\dots,\omega_n)\in\N^n,$ set
$
  T_\omega:=T_{\omega_1}\circ T_{\omega_2}\circ\cdots\circ T_{\omega_n}.$
Thus $T_\omega$ is the inverse branch of $T^n$ determined by $\omega$, and
\begin{equation}\label{eq:inverse-derivative}
|T_\omega'(x)|=|(T^n)'(T_\omega x)|^{-1}.
\end{equation}

This definition, a natural generalization of the Gauss map and its several modifications, follows the framework established by Jordan and Rams \cite{JR12} and further developed by Liao and Rams \cite{LR22}. The dynamical systems in this class include in particular the following examples:
 with $d=2$ include the Gauss map generating regular
continued fractions, the classical L\"uroth map
\cite{LJ}, and its alternating version \cite{KKK}.
The maps generating $N$-continued fractions \cite{BGKWY,DKN},
for each fixed integer $N\geq1$, are also $2$-decaying. Indeed,
their inverse branches satisfy
\[
  T_{N,q}(x)=\frac{N}{q+x},\qquad q\geq N,
  \qquad
  |T_{N,q}'(x)|=\frac{N}{(q+x)^2}\asymp q^{-2}.
\]

Another $2$-decaying example is the folded nearest-integer
continued-fraction map, corresponding to $\alpha=\frac12$ in
Nakada's family \cite{NH81,LMNN}. On $[0,\frac12]$, its inverse
branches are
\[
  x\longmapsto\frac1{a+x},\quad a\geq2,
  \qquad
  x\longmapsto\frac1{a-x},\quad a\geq3.
\]
Their absolute derivatives are comparable to $a^{-2}$, giving
a $2$-decaying system after ordering the branches and affinely
rescaling the interval to $[0,1]$.

An example with $d=3$ is the quadratic Gauss map
$T(x)=\{x^{-1/2}\}$, a particular case of R\'enyi's
$f$-expansions \cite{RA57}. Its inverse branches satisfy
\[
  T_j(x)=\frac1{(j+x)^2},
  \qquad
  |T_j'(x)|=\frac2{(j+x)^3}\asymp j^{-3}.
\]
Generalized L\"uroth maps \cite[Section~1.2]{BBDK}, including the
$\alpha$-L\"uroth maps of \cite{KMS}, provide examples for every
$d>1$, in particular $d=3$ and $d=4$, by choosing ordered partition
intervals satisfying $|I_j|\asymp j^{-d}$. Indeed, their inverse
branches are affine and satisfy $|T_j'|=|I_j|$. Several of these examples and their formulation as infinite iterated function systems are also discussed in
\cite[Section~1.3]{GHST}.

When one studies the geometric properties of a dynamical system, the first thing to check are its distortion properties. All the abovementioned variations of the Gauss map are smooth, and for that reason they satisfy the Bounded Distortion Property (BDP). In fact, in the literature there are already several theorems proved under an assumption that the $d$-decaying Gauss-like iterated function system satisfies the BDP, see for example \cite{LMSX,SZL25,SZ26}. One of our goals is to remove the necessity of this assumption: we introduce a new, weaker distortion property (Almost Tempered Distortion, ATD), we prove it is satisfied for all $d$-decaying systems, and then we show its usefulness by proving our main theorem. In particular, we prove the existence of the geometric pressure defined by cylinder lengths for $d$-decaying systems (which are outside the setting of Mauldin-Urba\'nski \cite{MU96} and Sarig, see \cite{SO99}).

Let us now get back to the geometric description of $d$-decaying systems. With the above definition, clearly there is a natural projection $\Pi:\N^{\N}\rightarrow[0,1]$ defined by 
$$  \Pi(\underline{a})=\lim_{n\to\infty}     T_{a_{1}}\circ\cdots\circ T_{a_{n}} (1)   $$
which gives the symbolic expansion $\underline{a}=(a_{1}(x), a_{2}(x),\cdots)   $ of any point $x\in[0,1].$ Such a symbolic expansion is uniquely defined except for countably many points. 
We will denote by $\Lambda=\Pi(\N^{\N}),$ the attractor of the IFS.
For any 
$(a_{1},a_{2},\cdots,a_{n})\in\N^{n},$  the set
$$ C_{n}[a_{1},\cdots,a_{n}]= T_{a_{1}}\circ\cdots\circ T_{a_{n}}([0,1]) $$ 
is called an $n$th level cylinder, it is the set of points in $[0,1]$ whose symbolic expansion begin with $a_{1},a_{2},\cdots, a_{n}.$ For simplicity denote $C[ a_{1}^{n}]:=C_{n}[a_{1},\cdots,a_{n}].$ We also write $C[\omega]=T_{\omega}([0,1]).$ Denote by $\left \vert I \right\vert$ the diameter of the interval $I.$
Let $D[a_1\ldots a_n]=\bigcup_{i\geq a_n} C[a_1\ldots a_{n-1} i].$

The chain rule and condition (v) in the definition of $d$-decaying Gauss like iterated function systems give, for every $\varepsilon>0$
\begin{equation}\label{fin}
K_{1}(\varepsilon)^{n}\prod_{i=1}^{n}a_{i}^{-d-\varepsilon}\leq\left\vert C[a_{1}^{n}] \right\vert\leq   K_{2}(\varepsilon)^{n}\prod_{i=1}^{n}a_{i}^{-d+\varepsilon}.   
\end{equation}

In the context of infinite iterated function systems, the digits $(a_n(x))_{n\geq 1}$ in the symbolic expansion of a point $x$ can attain arbitrarily large values. 
This naturally leads to the question of describing the sets of points for which the symbolic expansion grows at prescribed speed.
For example, digit-growth questions have also been studied directly for
$d$-decaying systems. Jordan and Rams \cite{JR12} studied
increasing-digit subsystems, Cao, Wang and Wu \cite{CWW13}
studied rapid digit growth, and Liao and Rams \cite{LR22}
studied large Birkhoff sums in this setting. Shi, Tan and Zhou
\cite{stz26} obtained a one-digit dimension theorem for general
thresholds under bounded distortion and a common orientation of
the branches. Zhang studied fast growth of geometric averages
\cite{Z25} and the lower and upper growth spectra of finite
weighted digit products on the $\log n$ scale \cite{Z251}.
Convergence exponents of digits and finite weighted products
were studied by Song and Zhang \cite{SZ26}.

%%%%%%%%%%%%%%%%%%%%%%%%%%%

For regular continued fractions a classical starting point is Good's work
\cite{Go_41}, in which he proved that
\[
\dim_H\left\{x\in(0,1)\setminus\mathbb Q:
 a_n(x)\longrightarrow\infty\text{ as }n\to\infty\right\}
=\frac12.
\]
Later, \L{}uczak \cite{Luczak} studied doubly exponential growth of partial quotients and showed that such rapid growth can lower the Hausdorff dimension below $1/2$. These results illustrate how restrictions on digit growth affect the size of the corresponding sets and motivate the study of more general growth conditions. For regular continued fractions, Wang and Wu \cite{WW08} determined the Hausdorff dimension for more general condition  $a_n(x)\geq\Phi(n)$ for infinitely
many $n$, with a prescribed threshold function $\Phi$. 

In the recent years in metric Diophantine approximation, the study of products of digits in the context of continued fraction has become an emerging topic, largely motivated by the work of Kleinbock and Wadleigh \cite{KlWa_18}. 
They showed that improvability of Dirichlet's Theorem is linked to the growth of products of consecutive continued fraction digits. There are several closely related developments.  Huang, Wu and Xu
\cite{HWX} subsequently obtained the Lebesgue measure and Hausdorff dimension
for products of an arbitrary finite number of consecutive partial quotients with all weights equal to one.
Other aspects of large digit products include the fast-growth
spectra studied by Bakhtawar and Feng \cite{BF}, the occurrence
of large products studied by Tan and Zhou \cite{TZ24}, also see \cite{SZ026} and the
Lebesgue-measure theory for functions of consecutive L\"uroth
digits studied by Kleinbock and Zhang \cite{KZ24}.

Going back to product of consecutive partial quotients, allowing different positive powers leads to the weighted-product
problem. Bakhtawar, Hussain, Kleinbock and Wang \cite{BHKW}
established the Lebesgue-measure criterion for an arbitrary number
of weighted factors and determined the Hausdorff dimension in the
regimes $B=1$ and $B=\infty$, where
$B=\liminf_{n\to\infty}\Phi(n)^{1/n}$.
For $1<B<\infty$, however, their exact dimension theorem treats
two factors, and the extension to more than two factors is discussed
in \cite[Section~8]{BHKW}. Brown-Sarre, Gonz\'alez Robert and Hussain
\cite{BAH23} developed the corresponding weighted theory for
L\"uroth expansions. Their dimension theorem in the regime
$1<B<\infty$ likewise treats two factors, and the question for
three or more factors is raised in \cite[Section~6]{BAH23}.

These questions motivate the dimension theorem below. It resolves
the weighted problem for any finite number of factors in the regime
$1<B<\infty$, even for consecutive digits in the classical Gauss
and L\"uroth systems, and also allows arbitrary fixed gaps between
the selected positions. 

Moreover, to avoid the necessity of proving the results separately for Gauss map, L\"{u}roth map, and so on, we investigate this set in the general setting of $d$-decaying Gauss-like iterated function systems. This involves developing the distortion theory of $d$-decaying Gauss-like systems, which might be of independent interest.  The passage to the general $d$-decaying
setting therefore accompanies a dimension result that is new
already in these classical examples.

\subsection{The main results}

Fix an integer $k\geq1.$ Let  $0\leq i_1<\cdots<i_k$ be a finite increasing sequence of integers, and let  $t_1,\ldots,t_k$ be positive real weights.  For $B\in(1,\infty)$, define
\[
E=E(B;i,t)=
\left\{
x\in [0,1]:
\prod_{m=1}^{k}a_{n+i_m}(x)^{t_m}\geq B^n
\text{ for infinitely many }n\in\N
\right\}.
\]
To state the dimension formula, let
\[
S=
\left\{
(b_1,\ldots,b_k)\in[0,\infty)^k:
\sum_{m=1}^{k}b_mt_m=1
\right\},
\]
be a simplex.

For $b\in S$, $s>1/d$, and $1\leq m\leq k$, put
\[
A_m(b,s)
=(d-1)sb_m+(ds-1)\sum_{j=1}^{m-1}b_j,
\]
and set
\begin{equation} \label{eqn:abs}
A(b,s) = \max_{1\leq m\leq k} A_m(b,s).
\end{equation}

The function $A(b,s)$ is positive, continuous with respect to $b$ and continuous and increasing with respect to $s$. Let 
\[
A(s)= \min_{b\in S} A(b,s).
\]

On $(1/d,\infty)$, $A(s)$ is positive, continuous and increasing,
and the minimum is attained because $S$ is compact.

\begin{thm}\label{MT}
Let $\{T_i\}_{i\geq1}$ be a $d$-decaying Gauss-like iterated
function system. The Hausdorff dimension of the set $E(B;i,t)$ is given by 
\[
\dim_H E(B;i,t)=s_0(B),
\qquad
s_0(B)=
\inf\left\{
s>\frac1d:P(s)\leq A(s)\log B
\right\},
\]
where $P$ denotes the pressure function defined in subsection \ref{psub}.
Moreover, $1/d\leq s_0(B)<1$.
\end{thm}

\begin{rem}
This result does not depend on the Bounded Distortion Property. While it is usual that we do not need BDP for statements about the set of points
with symbolic expansions escaping to infinity, as the cylinders in the geometric cover of such sets are superexponentially small, 
it is interesting that the same holds for $E(B;i,t)$. 
\end{rem}

The coordinates $b_m$ describe how the required exponential growth
is allocated among the selected digits. In the covering argument,
an earlier selected digit contributes the coefficient $ds-1$,
whereas the selected digit at which we use a tail interval contributes
$(d-1)s$. Thus the possible intermediate positions give the
exponents $A_m(b,s)$. The minimisation in \eqref{eqn:abs} selects
the most favourable allocation subject to the weighted-product
constraint. Coordinates $b_m=0$ are allowed: not every selected
digit must contribute exponential growth.

\begin{rem}[A Bowen-type formulation]
With geometric pressure interpreted as in Subsection~\ref{psub},
adding a constant to the potential gives
\[
P\bigl(T;-s\log|T'|-A(s)\log B\bigr)
=P(s)-A(s)\log B.
\]
Consequently, Theorem~\ref{MT} can be written as
\[
\dim_H E(B;i,t)=
\inf\left\{
s>\frac1d:
P\bigl(T;-s\log|T'|-A(s)\log B\bigr)\leq0
\right\}.
\]
\end{rem}

For completeness next we  consider the same problem with $B^{n}$ replaced by a general function
$\Phi:\N\to[1,\infty)$. Define
\[
E(\Phi;i,t)=
\left\{
x\in[0,1]:
\prod_{m=1}^{k}a_{n+i_m}(x)^{t_m}\geq\Phi(n)
\text{ for infinitely many }n\in\N
\right\}.
\]
Write
\begin{equation}\label{Bb}
\log B=\liminf_{n\to\infty}\frac{\log\Phi(n)}{n}.
\end{equation}
When $B=\infty$, also put
\[
\log b=\liminf_{n\to\infty}\frac{\log\log\Phi(n)}{n}.
\]

\begin{thm}\label{MT2}
Under the same hypotheses, for every $\Phi:\N\to[1,\infty)$,
\[
\dim_H E(\Phi;i,t)=
\begin{cases}
1, & B=1,\\[1ex]
s_0(B), & 1<B<\infty,\\[1ex]
\displaystyle\frac1{1+(d-1)b}, & B=\infty,
\end{cases}
\]
where the last expression is understood to be zero when $b=\infty$.
\end{thm}

For $1<B<\infty$, the dimension depends on the weights through
$A(s)$, whereas for $B=\infty$ it depends only on $d$ and $b$.
For a general function $\Phi$, the cases $B=1$ and $1<B<\infty$
follow from the result for $B^n$ by choosing the sequence $(n_j)$
in the lower-bound construction appropriately. For $B=\infty$,
we use the following \L{}uczak-type lemma.
Note that the below one-digit result does not require the Bounded
Distortion Property and therefore applies to the $d$-decaying
Gauss-like systems considered here.

\begin{lem}[One-digit double-exponential theorem]\label{luc}
Let $\{T_i\}_{i\geq1}$ be a $d$-decaying Gauss-like iterated
function system. For $a,b>1$ and a fixed integer $p\geq0$,
\begin{align*}
&\dim_H\left\{
x\in[0,1]:
a_{n+p}(x)\geq a^{b^n}
\text{ for all }n\in\N
\right\}\\
&\qquad=
\dim_H\left\{
x\in[0,1]:
a_{n+p}(x)\geq a^{b^n}
\text{ for infinitely many }n\in\N
\right\}
=\frac1{1+b(d-1)}.
\end{align*}
\end{lem}
\begin{proof}
For $p=0$, the result follows from \cite[Theorem~1.2]{CWW13}, together with the corresponding
lower-bound result of \cite{JR}; see also
\cite[Lemma~2.6]{stz26}. A fixed shift $p$ and finitely many
initial digits do not affect the Hausdorff dimension.
\end{proof}

\subsection{Applications of the dimension formula}

Theorem~\ref{MT2} contains the following Hausdorff-dimension results
as special cases.
\begin{itemize}

\item \emph{One continued-fraction digit: Wang--Wu \cite{WW08}.}
For the Gauss map, take $d=2$, $k=1$, $t_1=1$, and $i_1=0$.
Then $b_1=1$, so $A(s)=s$. Hence, when $1<B<\infty$,
\[
\dim_H\left\{x:a_n(x)\geq\Phi(n)
\text{ for infinitely many }n\in\N\right\}
=
\inf\left\{s>\frac12:P(s)\leq s\log B\right\}.
\]
Together with the cases $B=1$ and $B=\infty$ in
Theorem~\ref{MT2}, this recovers the theorem of Wang and Wu
\cite{WW08}.

\item \emph{Unweighted products of consecutive partial quotients:
Huang--Wu--Xu \cite{HWX}.}
For the Gauss map, take $d=2$, $t_1=\cdots=t_k=1$, and
$i_m=m-1$ for $1\leq m\leq k$. In this case,
\[
A(s)=\frac{(2s-1)s^k}{s^k-(1-s)^k},
\qquad \frac12<s<1.
\]
Thus Theorem~\ref{MT2} recovers the Hausdorff-dimension result
of Huang, Wu and Xu \cite{HWX} for
\[
\left\{x:a_n(x)\cdots a_{n+k-1}(x)\geq\Phi(n)
\text{ for infinitely many }n\in\N\right\}.
\]

\item \emph{Two weighted consecutive partial quotients:
Bakhtawar--Hussain--Kleinbock--Wang \cite{BHKW}.}
For the Gauss map, take $d=2$, $k=2$, $i_1=0$, $i_2=1$,
and arbitrary $t_1,t_2>0$. Minimising over
$t_1b_1+t_2b_2=1$ gives
\[
A(s)=\min\left\{
\frac{s}{t_2},\,
\frac{s^2}{st_1+(1-s)t_2}
\right\},\qquad \frac12<s<1.
\]
Consequently, Theorem~\ref{MT2} recovers the
Hausdorff-dimension formula of
\cite[Theorem~1.6]{BHKW} for
\[
\left\{x:a_n(x)^{t_1}a_{n+1}(x)^{t_2}\geq\Phi(n)
\text{ for infinitely many }n\in\N\right\}
\]
when $1<B<\infty$, as well as their dimension statements
for $B=1$ and $B=\infty$.

\item \emph{Weighted products of L\"uroth digits.}
For the classical L\"uroth map, using its usual digits $a\geq2$,
the inverse branches are
\[
T_a(x)=\frac1a+\frac{x}{a(a-1)}.
\]
They are affine and satisfy
\[
|T_a'(x)|=\frac1{a(a-1)}\asymp a^{-2},
\qquad \sup_{a\geq2,x\in[0,1]}|T_a'(x)|=\frac12.
\]
Thus this is a $2$-decaying system, with bounded-distortion
constant $1$. Its cylinder lengths and pressure are
\[
|C[a_1^n]|=\prod_{j=1}^n\frac1{a_j(a_j-1)},
\qquad
P(s)=\log\sum_{a=2}^{\infty}\frac1{[a(a-1)]^s}.
\]
It follows that, for $1<B<\infty$,
$\dim_HE(\Phi;i,t)$ is the unique $s\in(1/2,1)$ satisfying
\[
\sum_{a=2}^{\infty}\frac1{[a(a-1)]^s B^{A(s)}}=1,
\]
where $A(s)$ is evaluated with $d=2$ and the chosen weights.
For $k=2$ and $(i_1,i_2)=(0,1)$, the formula for $A(s)$
in the preceding item gives exactly the dimension theorem of
Brown-Sarre, Gonz\'alez Robert and Hussain
\cite[Theorem~1.6]{BAH23}.
For $k\geq3$, our theorem gives the corresponding formula for
arbitrary positive weights and arbitrary fixed positions.
The cases $B=1$ and $B=\infty$ also recover their respective
dimension statements.

\end{itemize}

The change from branch labels starting at $1$ to the usual
L\"uroth digits starting at $2$ does not affect the dimension
formula: the two weighted products differ by at most the fixed
factor $2^{t_1+\cdots+t_k}$, and multiplying $\Phi$ by a fixed
positive constant does not change $B$ or, when $B=\infty$, $b$.
The pressure is unchanged by relabelling the branches.

%%%%%%%%%%%%%%%%%%%%%%%%%%%%%%%%%%%%

\textbf{Organization of paper:} Computing the Hausdorff dimension of a set is typically accomplished in two steps: obtaining the upper and lower bounds separately.
The rest of the paper is organized as follows. In Section \ref{pre}, we prove some technical results which are the key to the proof of our main theorem, including the distortion properties and the construction of pressure function.  Sections \ref{UBS} and \ref{LBS} are devoted to the  proofs of  the upper bound  and the lower bound, respectively, for $\dim_H E(B;i,t) $ in Theorem \ref{MT}. Section \ref{genF}, proves Theorem \ref{MT2}.

\section{Preliminaries}\label{pre}

Throughout this section, $\N^n$ denotes the set of words of length $n$ over $\N$, and $\N^{\N}$ the set of infinite sequences. For $M\in\N$, the corresponding finite-alphabet spaces are $\Omega_{M}^{n}:=\{1,\ldots,M\}^n$ and $ \Omega_{M}^{\N}:=  \{1,\ldots,M\}^{\N}$. Given a finite or infinite sequence $\omega$, we write
\[
 \omega_\ell^r=(\omega_\ell,\omega_{\ell+1},\ldots,\omega_r).
\]

\subsection{Distortion}
For $\omega\in\N^n,$ we denote the distortion of a map $T_{\omega}$ by
\[
\Dist(T_{\omega}) := \sup_{x,y} \frac {|T_{\omega}'(x)|} {|T_{\omega}'(y)|}.
\]

There are several notions of distortion bounds applicable to iterated function systems:
\begin{itemize}
\item Bounded Distortion: there exists $K>1$ such that for all $n$, all $\omega \in\N^n$, and $x,y\in C[\omega]$
\[
K^{-1} \leq \frac {|T_\omega'(x)|} {|T_\omega'(y)|} \leq K,
\]
\item Tempered Distortion: there exists a function $\rho:\N\to [0,\infty), \frac 1n\rho(n)\to 0$, such that for all $n$, all $\omega \in\N^n$, and $x,y\in C[\omega]$
\[
e^{-\rho(n)} \leq \frac {|T_\omega'(x)|} {|T_\omega'(y)|} \leq e^{\rho(n)},
\]
\end{itemize}

We will add one more definition to the list: the iterated function system
{ satisfies \emph{Almost Tempered Distortion} if, for every $\varepsilon>0$, there is a function $\rho^{(\varepsilon)}:\N\to[0,\infty)$ with
\[
 \frac{\rho^{(\varepsilon)}(n)}n\longrightarrow0
\]
such that, for every $n\in\N$, every $\omega\in\N^n$, and every $x,y\in C[\omega]$,
\begin{equation}\label{eq:ATD}
 e^{-\rho^{(\varepsilon)}(n)}|T_\omega'(x)|^{\varepsilon}
 \leq \frac{|T_\omega'(x)|}{|T_\omega'(y)|}
 \leq e^{\rho^{(\varepsilon)}(n)}|T_\omega'(x)|^{-\varepsilon}.
\end{equation}
}

Equivalently, let $u,v\in\operatorname{int}C[\omega]$ and write $u=T_\omega(x)$ and $v=T_\omega(y)$. By \eqref{eq:inverse-derivative} above definition of almost tempered distortion \eqref{eq:ATD} is can be written as
\begin{equation}\label{eq:ATD-Tn}
 e^{-\rho^{(\varepsilon)}(n)}|(T^n)'(u)|^{-\varepsilon}
 \leq \frac{|(T^n)'(u)|}{|(T^n)'(v)|}
 \leq e^{\rho^{(\varepsilon)}(n)}|(T^n)'(u)|^{\varepsilon}.
\end{equation}

\begin{pro} \label{prop:dist}
The Gauss-like $d$-decaying iterated function system satisfies Almost Tempered Distortion. The Gauss-like $d$-decaying iterated function system restricted to finitely many symbols (that is, with $\omega\in \{1,\ldots,M\}^n$ for some fixed $M$) satisfies Tempered Distortion. If in addition we assume $T$ is $C^{1+{\rm Holder}}$ then it satisfies Bounded Distortion.
\end{pro}

\begin{proof}
Fix $\varepsilon>0$ and put
\[
K(\varepsilon)
=\max\{1,K_1(\varepsilon)^{-1},K_2(\varepsilon)\}.
\]
Choose $M\geq2$ sufficiently large that
$K(\varepsilon)<M^\varepsilon$.
Let $\omega\in\N^n$ and $x,y\in\operatorname{int}C[\omega]$.
By the chain rule,
\[
\log\frac{|(T^n)'(x)|}{|(T^n)'(y)|}
=
\sum_{i=0}^{n-1}
\log\frac{|T'(T^ix)|}{|T'(T^iy)|}=      \sum_{\substack{0\leq i\leq n-1\\ \omega_{i+1}\leq M}}\log\frac{|T'(T^ix)|}{|T'(T^iy)|}  +   \sum_{\substack{0\leq i \leq n-1\\ \omega_{i+1}> M}}\log\frac{|T'(T^ix)|}{|T'(T^iy)|}    .
\]
We divide the sum into the terms for which
$\omega_{i+1}\leq M$ and those for which $\omega_{i+1}>M$.

{For the first sum, $T^ix$ and $T^iy$ belong to the same cylinder
of length $n-i$. By uniform contraction, their distance tends
to zero exponentially fast as $n-i\to\infty$. Moreover, $\log|T'|$ is
uniformly continuous on each of the first $M$ branch intervals,
using one-sided values at endpoints. Thus, for any $\eta>0$,
the absolute value of each summand with $n-i$ sufficiently
large is at most $\eta$. There are only a bounded number of
remaining terms, and these are uniformly bounded. Consequently,
\[
\sum_{\substack{0\leq i\leq n-1\\ \omega_{i+1}\leq M}}
\left|\log|T'(T^ix)|-\log|T'(T^iy)|\right|
\leq \rho(n),
\qquad \frac{\rho(n)}n\longrightarrow0,
\]
uniformly in $\omega,x,y$.}

{For the second sum, condition~(v) and the choice of $M$ give

\[
\begin{aligned}
\max_{y\in C_{n}[x]} \sum_{\substack{0\leq i<n-1\\ \omega_{i+1}> M}}\log\frac{|T'(T^ix)|}{|T'(T^iy)|} &\leq  \sum_{\substack{0\leq i<n-1\\ \omega_{i+1}> M}}\log{|T'(T^ix)|}\frac{4\varepsilon}{d-2\varepsilon }\leq \frac{4\varepsilon}{d-2\varepsilon }  \sum_{\substack{0\leq i<n-1}}\log{|T'(T^ix)|}\\&=\frac{4\varepsilon}{d-2\varepsilon }\log{|(T^n)'(x)|}.
\end{aligned}
\]

Combining the two sums, we obtain
\[
\left|\log\frac{|(T^n)'(x)|}{|(T^n)'(y)|}\right|
\leq
\rho(n)+ \frac{4\varepsilon}{d-2\varepsilon }\log{|(T^n)'(x)|}    .
\]
Since $\frac{4\varepsilon}{d-2\varepsilon }   $ can be chosen to equal any prescribed
positive exponent, absorbing the constant $ \frac{4\varepsilon}{d-2\varepsilon }$
into $\rho(n)$ and exponentiating proves Almost Tempered
Distortion.}

For a fixed finite subsystem, every term belongs to the first
sum. Its logarithmic distortion is therefore bounded by
$\rho(n)=o(n)$, proving Tempered Distortion.

Finally, suppose that the functions $\log|T_i'|$ have a common
H\"older exponent and constant. Apply the chain rule to the
inverse branches. The two arguments in each term are images
under the same remaining composition, so their distances
decrease exponentially with its length. The H\"older estimates
therefore bound the logarithmic distortion by a convergent
geometric series, uniformly in $n$ and $\omega$. This proves
Bounded Distortion.
\end{proof}

\subsection{Pressure}\label{psub}

A central object in the thermodynamic formalism of dynamical systems is the topological pressure. In the context of conformal iterated function systems, it provides a powerful tool for analyzing the fractal geometry of the limit set. 
For infinite conformal iterated function systems (satisfying the Bounded Distortion Property), the existence and properties of the pressure function are well-established 
see Hanus-
Mauldin-Urba\'{n}ski \cite{HaMaUr_02}, Mauldin-Urba\'{n}ski \cite{MU96}, or their monograph \cite{MU}. These classical results guarantee that the limit defining the pressure exists and yields a convex, analytic function.

However, since our framework encompasses general $d$-decaying Gauss-like systems where BDP may fail, we cannot directly apply these classical theorems. 
 In particular, Mauldin and Urba\'{n}ski establish existence and further properties of pressure function under the acceptability condition, defined in \cite{MU} but this condition is in general not satisfied by our geometric potential.
Consequently, the existence of the pressure limit is no longer guaranteed a priori. Therefore, we must develop an independent approach to establish the foundational thermodynamics for our setting.

We now construct the pressure function and prove its key properties. For a real parameter $s,$ and a finite word $\omega=(\omega_{1},\cdots,\omega_{n})\in\N^{n},$ we define the potential 

\[
\phi_s(\omega) = s\log |C([\omega])|.\]
The $n$-th partition function is defined as 
\[
Z_n(s) = \sum_{|\omega|=n} e^{\phi_s(\omega)}=\sum_{|\omega|=n} |C([\omega])|^s. \]
and let 
\begin{equation}\label{pfin} P_n(s) = \frac 1n \log Z_n(s).
\end{equation}
\
Observe that $Z_n(s), P_n(s)$ exist and are finite for any $n$ when $s>1/d$. 
The topological pressure is given by the limit as follows:
\[P(s)=\lim_{n\to \infty}P_n(s) = \lim_{n\to\infty}\frac 1n \log Z_n(s).
\]

Our immediate goal is to prove the existence of this limit.
 We begin with establishing basic  properties of this topological pressure.
\begin{lem} \label{lem:convex}
The function $P_n$ is a convex, differentiable, and decreasing function of $s$ in the range $s\in (1/d,+\infty)$.
\end{lem}
\begin{proof}
The function $P_n$ is decreasing in $s$ since $|C([\omega])|<1.$ It therefore suffices to prove convexity and differentiability.
 For this, we analyze the derivatives of $Z_{n}(s)$. Let us define the following sums:
 
 $S_0=Z_n(s),$\quad $S_1=\sum_{|\omega|=n} |C([\omega])|^s \log |C([\omega])|,$\quad \text{ and } $S_2= \sum_{|\omega|=n} |C([\omega])|^s (\log |C([\omega]))^2$.

We have
\[
\frac d {ds} \log S_0 = \frac {dS_0/ds} {S_0} = \frac {S_1} {S_0}
\]
and
\[
\frac {d^2} {ds^2} \log S_0 = \frac d {ds}  \frac {S_1} {S_0}= \frac {S_0 dS_1/ds - S_1 dS_0/ds} {S_0^2} = \frac {S_0 S_2 - S_1^2} {S_0^2}.
\]
Obviously, $S_0^2>0$. As for the numerator,
\begin{align*}
S_0 S_2 - S_1^2 &= \sum_{|\omega|=n} \sum_{|\tau|=n} |C([\omega])|^s |C([\tau])|^s ((\log |C([\omega])|)^2 + (\log |C([\tau])|)^2 -2 \log |C([\omega])| \log |C([\tau])|))\\& =  \sum_{|\omega|=n} \sum_{|\tau|=n} |C([\omega])|^s |C([\tau])|^s( \log |C([\omega])|- \log |C([\tau])|)^2 \geq 0.
\end{align*}
Thus, $\log Z_n(s)$ is a convex differentiable function of $s$. Consequently it follows that $P_{n}(s)$ is also convex differentiable.
\end{proof}

We now prove a uniform Lipschitz estimate. This uniformity is essential for the  proof of the existence of pressure function.
\begin{lem} \label{lem:unifcont} 
\text
For every $s>1/d,$ there exist $z=z(s)$ such that for every $\delta<(s-1/d)/2$ and for every $n,$ we have
\[
|P_n(s-\delta)-P_n(s)| \leq z\delta
\text{ and }
|P_n(s+\delta)-P_n(s)| \leq z\delta.
\]
\end{lem}
\begin{proof}
Fix $s>1/d$. Let $s_0= (s+1/d)/2$ and $\varepsilon<d-1/s_0$. We have
\[
P_n(s_0) \leq \log \sum_i (K_2(\varepsilon) i^{-d+\varepsilon})^{s_0} =: R_1
\]
and
\[
P_n(s) \geq \log \sum_i (K_1(\varepsilon) i^{-d-\varepsilon})^s =: R_2.
\]
for every $t\in [s_0,s]$ we have 
\[
P_n(t)-P_n(s) \leq (s-t) \frac {R_1-R_2}{s-s_0}
\]
and for every $t>s$ we have
\[
P_n(s)-P_n(t) \leq {(t-s)} \frac {R_1-R_2}{s-s_0}.
\]
As the function $P_n$ is decreasing, this implies the assertion.
\end{proof}

\begin{pro}\label{prop:press}
The limit 
\[
P(s) = \lim_{n\to\infty} P_n(s)
\]
exists for every $s>1/d$.
\end{pro}
\begin{proof}
Fix $s>1/d$ and $\varepsilon>0$ so small such that 
%$\varepsilon < s-1/d$ 
\begin{equation}\label{eq:epscond}
  s\varepsilon<\tfrac12\big(s-\tfrac1d\big)
  \qquad\text{and}\qquad
  s(1-\varepsilon)>\tfrac1d .
\end{equation}
Fix a large $n$. For $k\in\N,$ we have
\[
Z_{kn}(s) = \sum_{|\omega^{(1)}|=n} \ldots \sum_{|\omega^{(k)}|=n} |C([\omega^{(1)} \ldots \omega^{(k)}])|^s.
\]
Thus, 
by the chain rule together with \eqref{eq:ATD},
\[
Z_{kn}(s) \leq e^{k \rho^{(\varepsilon)}(n)} \cdot \sum_{\omega^{(1)}} |C([\omega^{(1)}])|^{s(1-\varepsilon)} \cdot\ldots\cdot \sum_{\omega^{(k)}} |C([\omega^{(k)}])|^{s(1-\varepsilon)}
\]
and
\[
Z_{kn}(s) \geq e^{-k \rho^{(\varepsilon)}(n)} \cdot \sum_{\omega^{(1)}} |C([\omega^{(1)}])|^{s(1+\varepsilon)} \cdot\ldots\cdot \sum_{\omega^{(k)}} |C([\omega^{(k)}])|^{s(1+\varepsilon)}.
\]

Writing an arbitrary $m$ as $m=kn+j$ with $0\le j<n$, and noting that
$Z_{kn+j}(s)/Z_{kn}(s)$ lies between $c(n,s)^{-1}$ and $c(n,s)$ for a constant
$c(n,s)$ independent of $k$, we obtain
\[
\limsup_{m\to\infty} \frac 1m \log Z_m(s) \leq \frac 1n \rho^{(\varepsilon)}(n) + P_n(s(1-\varepsilon))
\]
and
\[
\liminf_{m\to\infty} \frac 1m \log Z_m(s) \geq - \frac 1n \rho^{(\varepsilon)}(n) + P_n(s(1+\varepsilon)).
\]
Therefore, by Lemma \ref{lem:unifcont}
\[
\limsup_{m\to\infty} \frac 1m\log Z_m(s) -  \liminf_{m\to\infty} \frac 1m\log Z_m(s)\leq \frac 2n  \rho^{(\varepsilon)}(n) + P_n(s(1-\varepsilon))- P_n(s(1+\varepsilon))\leq  \frac 2n  \rho^{(\varepsilon)}(n) + 2z\varepsilon
\]
and the right hand side can be made arbitrarily small by choosing sufficiently small $\varepsilon$ and then sufficiently large $n$.
\end{proof}

The pressure function $P(s),$ whose existence is guaranteed by Proposition~\ref{prop:press}, inherits the fundamental properties of the approximants $P_n (s)$ and plays a central role in the dimension theory of the attractor  of the iterated function system. {From the above discussions we conclude that.}
\begin{pro}[Properties of Pressure Function]\label{PPF}
The pressure function $P: (1/d, \infty) \to \mathbb{R}$ satisfies the following:
\begin{itemize}
\item[i.]  $ P(s)$ is strictly decreasing and strictly convex.
\item[ii.] $P$ is finite for $s>1/d.$
\item[iii.]    $\lim_{s \to \infty} P(s) = -\infty$.
\item[iv.] P(1)=0.
\end{itemize}
\end{pro}
\begin{proof}
(i) follows from Lemma \ref{lem:convex} and Proposition \ref{prop:press}. (ii) follows from \eqref{fin}. Also $P(t)\to\infty$ as $t\to\infty.$ Further since the level $n$ cylinders partition $[0,1]$ so $Z_{n}(1)=1$ and $P(1)=0.$
\end{proof}

We now consider finite subsystems. For $M\in\N,$ recall $\Omega_{M}^{n}=\{1,\cdots,M\}^{n}.$ Define the  
\begin{equation}\label{finP}
P_{M}(s)=\lim_{n\to \infty}P_{M,n}(s) = \lim_{n\to\infty}\frac 1n \log Z_{M,n}(s), \quad \text{ where }
Z_{M,n}(s)=\sum_{\omega\in\Omega_{M}^{n}}|C[\omega]|^s.
\end{equation}
Therefore, by Proposition \ref{prop:press}, the limit in the definition of pressure function $P_{M}$ is guarantee. 

The following proposition states that in $d$-decaying Gauss-like iterated function systems, the pressure function has a continuity property when the full system is approximated by its finite-alphabet subsystems.
\begin{pro}\label{CA}
For any $s > 1/d$, the pressure is continuous with respect to the alphabet, i.e.,
\begin{equation*}
P(s)  = \lim_{M \to \infty} P_M(s).
\end{equation*}
\end{pro}
\begin{proof}
By definition of Pressure function $P_{M,n}(s)\leq P_{n}(s).$ Hence 
$\lim_{n\to\infty}P_{M,n}(s)\leq \lim_{n\to\infty}P_{n}(s).$ This implies  $\lim_{M \to \infty}P_M(s)\leq P(s),$ for every $M.$ Thus $$\lim_{M \to \infty}P_M(s)\leq P(s).$$

To prove the reverse inequality, fix some $n$ and $\varepsilon$ and observe that for every $M$ we have
\[
\liminf_{m\to\infty} \frac 1m \log Z_{M,m}(s) \geq -\frac 1n \rho^{(\varepsilon)}(n) +P_{M,n}(s(1+\varepsilon)).
\]
Indeed, we have proven this inequality for $Z_m$ in the course of proving Proposition \ref{prop:press}, and the same proof works for every $Z_{M,m}$. This implies
\[
\lim_{M\to\infty} P_M(s) = \lim_{M\to\infty}  \lim_{m\to\infty} \frac 1m \log Z_{M,m}(s)\geq  -\frac 1n \rho^{(\varepsilon)}(n)+\lim_{M\to\infty} P_{M,n}(s(1+\varepsilon))=  -\frac 1n \rho^{(\varepsilon)}(n) + P_n(s(1+\varepsilon)).
\]
Keeping $\varepsilon$ fixed and passing with $n$ to infinity we get
\[
\lim_{M\to\infty} P_M(s) \geq \lim_{n\to\infty}  -\frac 1n \rho^{(\varepsilon)}(n) + \lim_{n\to\infty}  P_n(s(1+\varepsilon))= P(s(1+\varepsilon)).
\]
Passing now with $\varepsilon$ to zero, we get
\[
\lim_{M\to\infty} P_M(s) \geq \lim_{\varepsilon\searrow 0}  P(s(1+\varepsilon))=P(s).
\]
Combining both inequalities we obtained the required result.
\end{proof}
We now connect the pressure to the critical exponent \( s_0(B) \) from our main theorem. Let $ M \in \mathbb{N}.$  Recall $s_0(B)= \inf \{ s: P(s) \leq A(s) \log B \},$  and define
\begin{align*}
s_{0,n}(B) = \inf \{ s: P_{n}(s) \leq A(s) \log B \},\quad
s_{0,M}(B) = \inf \{ s: P_M(s) \leq A(s) \log B \},
\end{align*}
where $s_{0,n}(B)$ uses the pressure at a finite level that is  $P_{n}(s)$ defined in \eqref{pfin} and
$s_{0,M}(B)$ uses the pressure of the finite subsystem $P_M(s)$ defined in \eqref{finP}.
Thus it follows from the definition of the pressure function and by Proposition \ref{CA} that
\begin{lem} For any $M\in\N,$ we have
\begin{equation*}
 \lim_{n\to\infty}  s_{0,n}(B)=s_{0}(B),       \quad    \lim_{M \to \infty} s_{0,M}(B) = s_0(B).
\end{equation*}
Moreover, $s_0$(B) is continuous and is strictly decreasing as a function of $B > 1$ as long as $s>1/d.$
\end{lem}
 The function $s_0(B)$ is decreasing because $P(s)$ is decreasing in $s$ and $A(s)\log B$ is increasing in $B$. Continuity follows from the continuity and monotonicity of both $P(s)$ and $A(s)$.

\begin{pro}\label{dimB}
As a function of  $B>1,$  $1/d\leq s_0(B)<1,$ and satisfies.
\[
\lim_{B \to \infty} s_0(B) = \frac{1}{d} \quad \text{and} \quad \lim_{B \to 1} s_0(B) = 1.
\]
\end{pro}
\begin{proof}
 As $B\to\infty,$ for any fixed $s > 1/d$, by Proposition \ref{PPF} (ii) $P(s)$ is finite while $A(s) \log B \to \infty$. Thus, for sufficiently large $B$,  $P(s) \leq A(s) \log B.$ So $s_{0}(B)\leq s$ for all $s<1/d,$ which implies $\lim_{B\to\infty}s_{0}(B)=1/d.$

As $B \to 1$, we have  $A(s) \log B \to 0$. The condition in the infimum becomes $P(s) \leq 0$. By Proposition \ref{PPF}, (i) and (iv), $P(1)=0$ and $P(s)$ is strictly decreasing, so the unique solution of $P(s)=0$ is $s=1.$ Hence $s_{0}(B)\to1$ as $B\to 1$
\end{proof}

\subsection{$s$-measures}

In the main body of the paper we will need the following construction. Fix some large $M\in\N$ and consider the restriction of the system to the maps $T_1,\ldots, T_M$. For any $s\in \R$ we will distribute a measure on the limit set in such a way that for every $n$ and every $\omega\in\Omega_{M}^{n}$ the measure ${\nu}(C[\omega])$ is roughly proportional to $|C[\omega]|^s$.

We will use the {\it conformal measures}, which in our case means probabilistic measures on the symbolic space $\Omega^{\N}_M$ such that for certain continuous potential $\phi$ we have
\[
\nu(TA) = \int_A e^{\phi(\omega)} \nu(d\omega)
\]
for every measurable set $A$ such that $TA$ is also measurable and $T|_A$ is invertible. In other words, the conformal measure is a measure with prescribed Jacobian. By \cite[Theorem 3.12]{DU}, for any $M$ and any continuous function $\varphi$ on $\Omega^{\N}_M$, there exists a $(P_M(\varphi)-\varphi)$-conformal measure on $\Omega^{\N}_M$.

We are going to apply this to the geometric potentials $\varphi_s(\omega) = -s\log |T'(\Pi(\omega))|$ for $s\in [0,\infty)$. By the choice of $\varphi_s$, we have thus a probabilistic measure on $\Pi(\Omega^{\N}_M)$ such that for any $n$ and for any word $\omega\in \Omega_{M}^n$ we have
\[
e^{-n P_M(\varphi_s)}\cdot \min_{x\in C[\omega]}|(T^n)'(x)|^{-s} \leq \nu(C[\omega]) \leq e^{-n P_M(\varphi_s)}\cdot \max_{x\in C[\omega]}|(T^n)'(x)|^{-s}. 
\]
We now remind that by Proposition \ref{prop:dist} we have Tempered Distortion. Thus,
\[
\frac {\max_{x\in C[\omega]}|(T^n)'(x)|^{-s}} {\min_{x\in C[\omega]}|(T^n)'(x)|^{-s}} < e^{\rho(n)}
\]
for certain sublinear function $\rho$. Thus,
\[
\max_{x\in C[\omega]}|(T^n)'(x)|^{-s} \leq e^{\rho(n)} \cdot |C[\omega]|^s
\]
and
\[
\min_{x\in C[\omega]}|(T^n)'(x)|^{-s} \geq e^{-\rho(n)} \cdot |C[\omega]|^s
\]
We obtained, for any $s\in [0,\infty)$ and $M$, a probabilistic measure $\nu_s$ satisfying
\begin{equation} \label{eqn:smeasure}
e^{-\rho(n)} \leq \frac {\nu_s(C[\omega])} {|C[\omega]|^s \cdot e^{-nP_M(s)}} \leq e^{\rho(n)}
\end{equation}
for certain sublinear function $\rho$, for all $n$ and all words $\omega\in \{1,\ldots,M\}^n$. We will call it {\it $s$-measure}.

In the lower-bound proof we choose $P_M(s)>0$, so the upper bound in~\eqref{eqn:smeasure} tends to zero uniformly in $\omega$. Notice also that
\begin{equation}\label{eq:finite-cylinder-bound}
\nu_s(C[\omega])\leq K|C[\omega]|^s
\quad\text{if }P_M(s)>0,
\end{equation}
because $\sup_{n\geq0}\exp(\rho(n)-nP_M(s))<\infty$. We will apply this cylinder estimate during the finite free blocks in the lower bound estimate.

\section{Proof of the Upper bound for $\dim_{H}E(B;i,t)$ }\label{UBS}

This section is devoted to establishing the upper bound \( \dim_H E(B;i,t)\leq s_0(B).\)

Fix $s\in(s_0(B),1)$ and put \( A=A(s).\) Choose small parameters $\delta,\varepsilon>0$ such that \( s(1-\varepsilon)>\frac1d.\)
For every sufficiently large $n\in\N$, define
\[
  E_n = \left\{x\in[0,1): \prod_{m=1}^{k}a_{n+i_m}^{t_m}(x)\geq B^n \right\}.
\]
Since \(E(B;i,t)=\limsup_{n\to\infty}E_n,\)
it suffices to construct, for every sufficiently large $n$, a cover
$\{D_i\}$ of $E_n$ such that \(  \max_i|D_i|\longrightarrow0\) and
\( \sum_i|D_i|^s\leq c\gamma^n\)
for some $\gamma<1$. Indeed, these estimates imply \(\mathcal H^s(E(B;i,t))=0\)
and hence \( \dim_HE(B;i,t)\leq s.\)
Letting $s\rightarrow s_0$ will then complete the proof.

Our cover is a union of $k$ families,
\[
  \{D_i\} =\bigcup_{t=1}^{k}\{D_i^{(t)}\},
\]
where each family $\{D_i^{(t)}\}$ consists of intervals of the form
\(D[a_1,\ldots,a_{n+i_t}]=: D[a_1^{n+i_t}],\)
chosen according to restrictions on the special digits \(a_{n+i_1},\ldots,a_{n+i_t}.\)

\subsection{Construction of the covers $\{D_i^{(t)}\}$}

Divide the possible values of each special digit into the blocks \( [1,B^{n\delta}]\cap\N\)
and \(\big(B^{n(\ell-1)\delta},B^{n\ell\delta}\big]\cap\N,\ell\geq2.\)
The $\ell$th block contains at most \( B^{n\ell\delta}\) integers, and the values in one block differ by a factor at most $B^{n\delta}$.

For every $b=(b_1,\ldots,b_k)\in S$, we have
\( \max_{1\leq m\leq k}A_m(b,s) =A(b,s) \geq A(s)=A.\)
Let $t$ be the smallest index such that \( A_t(b,s)\geq A.\) For $1\leq m\leq t$, put
\[\ell_m(b) = \max\left\{1,\left\lceil\frac{b_m}{\delta}\right\rceil \right\}.\]
Then \( 0\leq\delta\ell_m(b)-b_m\leq\delta.\)
For each fixed $t\in\{1,\ldots,k\}$, let
\[ G_t
  =
  \left\{
  \big(\ell_1(b),\ldots,\ell_t(b)\big):
  b\in S
  \text{ and }t\text{ is the smallest index such that }
  A_t(b,s)\geq A
  \right\}.
\]
Since $b_m\leq1/t_m$ for every $b\in S$, there are at most
$c/\delta$ possible values of each $\ell_m(b)$. Therefore,
\[
  \#G_t\leq c\delta^{-t}.
\]

For every \( (\ell_1,\ldots,\ell_t)\in G_t,\) we allow $a_{n+i_m}$ to run through the $\ell_m$th block for $m<t$, while at the last special position $n+i_t$ we use only the
smallest integer of the $\ell_t$th block. Let $F_t$ be the union of
all the resulting tuples \( (a_{n+i_1},\ldots,a_{n+i_t}).\)
For all sufficiently large $n$, every block occurring in this
definition is non-empty.

The family $\{D_i^{(t)}\}$ consists of all intervals \(D[a_1^{n+i_t}]\) such that:

\begin{itemize}
\item[(i)]
all the digits $a_p$ with
\[
  p\leq n+i_t,
  \qquad
  p\notin\{n+i_1,\ldots,n+i_t\},
\]
are arbitrary;

\item[(ii)]
\[
  (a_{n+i_1},\ldots,a_{n+i_t})\in F_t.
\]
\end{itemize}

At the last special position, one integer per block is sufficient.
Indeed, if $u$ is the smallest integer of the corresponding block,
then, by the definition of $D$,
\[
  D[a_1,\ldots,a_{n+i_t-1},u]
  =
  \bigcup_{a\geq u}
  C[a_1,\ldots,a_{n+i_t-1},a].
\]
Therefore, this interval contains every cylinder whose digit at the
position $n+i_t$ lies in that block.

\subsection{The families cover $E_n$}

Let $x\in E_n$, and let $r$ be the smallest index such that
\[
  \prod_{m=1}^{r}a_{n+i_m}^{t_m}(x)\geq B^n.
\]
For $m<r$, put \( b_m= \frac{\log a_{n+i_m}(x)}{n\log B}.\)
By the minimality of $r$, \( \sum_{m<r}t_mb_m<1.\) Put
\[ b_r=\frac{1-\sum_{m<r}t_mb_m}{t_r}>0,\qquad b_{r+1}=\cdots=b_k=0.
\]
Then \( b=(b_1,\ldots,b_k)\in S.\)
Moreover, \( B^{nb_r}\leq a_{n+i_r}(x).\)

Let $t$ be the smallest index such that \(  A_t(b,s)\geq A.\) We claim that \(  t\leq r.\)
If $r=k$, this is immediate. Suppose that $r<k$. Since \( b_{r+1}=\cdots=b_k=0,\)
we have
\[ A_{r+1}(b,s)=\cdots=A_k(b,s),\]
and
\[
\begin{aligned}
  A_{r+1}(b,s)-A_r(b,s)&=(ds-1)b_r-(d-1)sb_r=(s-1)b_r \leq0,
\end{aligned}
\]
because $s<1$. Hence the maximum of \(  A_1(b,s),\ldots,A_k(b,s)\)
is attained at an index not exceeding $r$. Since this maximum is at
least $A$, the smallest index $t$ satisfying $A_t(b,s)\geq A$ must
satisfy $t\leq r$.

For $m<t$, \(a_{n+i_m}(x)\) belongs to the $\ell_m(b)$th block.
Let $a'_{n+i_t}(x)$ be the smallest integer in the
$\ell_t(b)$th block. We claim that \(a'_{n+i_t}(x)\leq a_{n+i_t}(x).\)

If $t<r$, then \(b_t = \frac{\log a_{n+i_t}(x)}{n\log B},\)
so $a_{n+i_t}(x)$ itself belongs to the $\ell_t(b)$th block.

If $t=r$ and $\ell_r(b)=1$, then \(  a'_{n+i_r}(x)=1\leq a_{n+i_r}(x).\)

If $t=r$ and $\ell_r(b)\geq2$, then \((\ell_r(b)-1)\delta<b_r,\)
and hence \(B^{n(\ell_r(b)-1)\delta}< B^{nb_r} \leq a_{n+i_r}(x).\)

Thus, the smallest integer in the $\ell_r(b)$th block is again not
larger than $a_{n+i_r}(x)$.

Consequently,
\[
  \big(
  a_{n+i_1}(x),\ldots,
  a_{n+i_{t-1}}(x),a'_{n+i_t}(x)
  \big)
  \in F_t.
\]
Moreover,
\[
\begin{aligned}
x\in C[
a_1(x),\ldots,a_{n+i_t}(x) ]\subseteq D[
a_1(x),\ldots,a_{n+i_t-1}(x),a'_{n+i_t}(x)
].
\end{aligned}
\]
Therefore, \( E_n \subseteq \bigcup_{t=1}^{k}\bigcup_iD_i^{(t)}.\)

\subsection{Estimating the Size of a Covering Element}\

By Almost Tempered Distortion \ref{eq:ATD}  and condition~(v), with the
decay error chosen sufficiently small in terms of $\varepsilon$,
we obtain the following estimate on the size of each element in the cover $\{D_{i}^{(t)}\}$  after enlarging
$\rho^{(\varepsilon)}$ to absorb fixed factors.

\[
|D_i^{(t)}| \leq e^{\rho^{(\varepsilon)}(n+i_t)} \cdot \left(|C[a_1^{n+i_1-1}]| \cdot |C[a_{n+i_1+1}^{n+i_2-1}]| \cdot \ldots \cdot |C[a_{n+i_{t-1}+1}^{n+i_t-1}]| \cdot \prod_{m=1}^{t-1} a_{n+i_m}^{-d} \cdot a_{n+i_t}^{1-d}\right)^{1-\varepsilon}.
\]
To estimate $\sum |D_i^{(t)}|^s,$ we analyze the structure of the cylinders. The sum is taken over all cylinders in the family $\{D_{i}^{t} \},$ which corresponds to the sequences where the special digits lie in $F_{t}:$
$$   \sum |D_i^{(t)}|^s     =\sum_{\substack{ a_{1},\cdots,a_{n+i_{t}}\\(a_{n+i_{1}},\cdots,a_{n+i_{t}})\in F_{t}}} |D[a_{1}^{n+i_{t}}]|^s.      $$

We decompose this sum according to the positions of the special digits. First, observe that
\[
\sum_{a_{1},\cdots,a_{n+i_{1}-1}} |C[a_1^{n+i_1-1}]|^{s(1-\varepsilon)}  = Z_{n+i_1-1}(s(1-\varepsilon))
\]
which is, up to a multiplicative constant, equal to $Z_n(s(1-\varepsilon))$. Moreover, for the intermediate blocks between special digits, for any fixed prefix $a_{1},\cdots, a_{n+i_{m}},$ the sums
\[
\sum_{a_{n+i_{m}+1},\cdots,a_{n+i_{m+1}-1 }} |C[a_{n+i_m+1}^{n+i_{m+1}-1}]|^{s(1-\varepsilon)}
\]
are all bounded by a constant independent of $n.$ Therefore for distortion property, we can write
\[
\sum |D_i^{(t)}|^s{ \approx} e^{\rho^{(\varepsilon)}(n{+i_{t}})} \cdot Z_{n{+i_{1}-1}}(s(1-\varepsilon )) \cdot \sum_{(a_{n+i_{1}},\cdots,a_{n+i_{t}})\in F_t} \prod_{m=1}^{t-1} a_{n+i_m}^{-sd(1-\varepsilon)} \cdot a_{n+i_t}^{s(1-d)(1-\varepsilon)}.
\]

Here the sums over the digits between two consecutive special
positions are bounded by a constant independent of $n$. Indeed, the
lengths of these intermediate blocks are fixed, and \(s(1-\varepsilon)>\frac1d.\)

 We now estimate
\[
\sum_{(a_{n+i_{1}},\cdots,a_{n+i_{t}})\in F_t} \prod_{m=1}^{t-1} a_{n+i_m}^{-sd(1-\varepsilon)} \cdot a_{n+i_t}^{s(1-d)(1-\varepsilon)}
\]

\textbf{Step 1:} \textbf{A fixed element of $G_{t}.$}

Fix $ (\ell_1,\dots,\ell_t) \in G_t .$
For each $m\leq t-1$, the number of integers in block \( \ell_m \) is at most \( B^{n\ell_m\delta} \).
Every integer \( a_{n+i_{m}} \) in this  block satisfies:\\
 \[a_{n+i_{m}}^{-sd(1-\varepsilon)} \leq \left(B^{n(\ell_m-1)\delta}\right)^{-sd(  1-\varepsilon   )} = B^{-n(\ell_m-1)\delta \cdot sd(1-\varepsilon)}.\]

 If \( m = t \),
 that is at the position $n+i_t$, only the smallest integer of the
$\ell_t$th block occurs, and it satisfies
   \[ a_{n+i_{t}}^{s(1-d)(1-\varepsilon)} \leq \left(B^{n(\ell_t-1)\delta}\right)^{s(1-d)(1-\varepsilon)} = B^{n(\ell_t-1)\delta\cdot s(1-d)(1- \varepsilon  )} \].\\
Therefore, for fixed \( (\ell_1,\dots,\ell_t) \), the contribution is at most
\[
 \prod_{m=1}^{t-1} \left( B^{n\ell_m\delta} \right) \cdot \prod_{m=1}^{t-1} B^{-n(\ell_m-1)\delta \cdot sd(1-\varepsilon)} \cdot B^{n(\ell_t-1)\delta \cdot s(1-d)(1-\varepsilon)}.
\]\\
\textbf{Step 2:} \textbf{Summation over \( G_t \).}
\begin{multline*}
  \sum_{(a_{n+i_{1}},\cdots,a_{n+i_{t}})\in F_t} \prod_{m=1}^{t-1} a_{n+i_m}^{-sd(1-\varepsilon)} \cdot a_{n+i_t}^{s(1-d)(1-\varepsilon)}         \leq 
   \sum_{(\ell_1,\dots,\ell_t) \in G_t} \left[ \prod_{m=1}^{t-1} B^{n\ell_m\delta} \cdot \prod_{m=1}^{t-1} B^{-n(\ell_m-1)\delta \cdot sd(1-\varepsilon)} \cdot B^{n(\ell_t-1)\delta\cdot s(1-d)(1-\varepsilon)} \right]\\
\leq
 \sum_{(\ell_1,\dots,\ell_t) \in G_t} \left[  B^{-n\delta[(sd-1) \sum_{m=1}^{t-1} \ell_{m}+s\ell_{t}(d-1) ] }\cdot  B^{n\delta[    \sum_{m=1}^{t-1} \varepsilon \ell_{m}sd+ \varepsilon\ell_{t}s(d-1) + \sum_{m=1}^{t-1}  sd( 1- \varepsilon   )+s(d-1) -s(d-1) \varepsilon        ] }    \right]\\
 =  \sum_{(\ell_1,\dots,\ell_t) \in G_t} \left[  B^{-n\delta[(sd-1) \sum_{m=1}^{t-1} \ell_{m}+s\ell_{t}(d-1) ] }\cdot  B^{n\delta[  \varepsilon sd  \sum_{m=1}^{t-1} \ell_{m}+ \varepsilon\ell_{t}s(d-1) + (t-1)  sd( 1- \varepsilon   )+s(d-1)(1- \varepsilon)        ] }    \right]    .
\end{multline*}
Note that since $\varepsilon$ is arbitrary small, and \(\delta\ell_m\leq b_m+\delta\leq\frac1{t_m}+\delta,\)
 the $\varepsilon$ and constant dependent part of the exponent in  
$$B^{n\delta[  \varepsilon sd  \sum_{m=1}^{t-1} \ell_{m}+ \varepsilon\ell_{t}s(d-1) + (t-1)  sd( 1- \varepsilon   )+s(d-1)(1- \varepsilon)        ] }$$   is small compared to the main  term and is bounded by $B^{n(c_{1}\delta+c_{2}\epsilon)}$ for some constants $c_{1},c_{2}>0,$ independent of $n,\delta,\varepsilon.$
Thus we can write up to factor $k(\delta,\varepsilon)\cdot B^{n(c_1\delta+c_2\varepsilon)},$ 
\begin{align*}
  \sum_{(a_{n+i_{1}},\cdots,a_{n+i_{t}})\in F_t} \prod_{m=1}^{t-1} a_{n+i_m}^{-sd(1-\varepsilon)} \cdot a_{n+i_t}^{s(1-d)(1-\varepsilon)}   \leq&
k(\delta,\varepsilon)B^{n(c_1\delta+c_2\varepsilon)}
\sum_{(\ell_1,\ldots,\ell_t)\in G_t}
B^{-n\delta\left[
(sd-1)\sum_{m=1}^{t-1}\ell_m+s(d-1)\ell_t
\right]}.  
\end{align*}

Step 3: \textbf{Apply the maximum bound defining $A.$}

By the definition of $A_t$,
\[
\begin{aligned}
\delta\left[(sd-1)\sum_{m=1}^{t-1}\ell_m+s(d-1)\ell_t
\right]=A_t\big((\delta\ell_1,\ldots,\delta\ell_t,0,\ldots,0),s
\big).\end{aligned}\]
Therefore,
\[
\begin{aligned}
&\sum_{(\ell_1,\ldots,\ell_t)\in G_t}
B^{-nA_t(
(\delta\ell_1,\ldots,\delta\ell_t,0,\ldots,0),s
)}
\\
&\leq
\#G_t
\max_{(\ell_1,\ldots,\ell_t)\in G_t}
B^{-nA_t(
(\delta\ell_1,\ldots,\delta\ell_t,0,\ldots,0),s
)}
\\
&=
\#G_t
B^{-n\min\limits_{(\ell_1,\ldots,\ell_t)\in G_t}
A_t(
(\delta\ell_1,\ldots,\delta\ell_t,0,\ldots,0),s
)}.
\end{aligned}
\]

Every $(\ell_1,\ldots,\ell_t)\in G_t$ comes from a point $b\in S$
such that \( A_t(b,s)\geq A. \) Since \(\delta\ell_m\geq b_m\) and, as $s>1/d$, both coefficients $ds-1$ and $(d-1)s$ are positive, we have
\[
\begin{aligned}
A_t\big( (\delta\ell_1,\ldots,\delta\ell_t,0,\ldots,0),s\big)\geq A_t(b,s)\geq A.
\end{aligned}
\]
Consequently,
\[
\min_{(\ell_1,\ldots,\ell_t)\in G_t}
A_t\big(
(\delta\ell_1,\ldots,\delta\ell_t,0,\ldots,0),s
\big)
\geq A.
\]
Using also \(\#G_t\leq c\delta^{-t},\) we finally obtain
\[
\begin{aligned}
&\sum_{(a_{n+i_1},\ldots,a_{n+i_t})\in F_t}
\prod_{m=1}^{t-1}a_{n+i_m}^{-sd(1-\varepsilon)}
\cdot a_{n+i_t}^{s(1-d)(1-\varepsilon)}\leq c(\delta,\varepsilon)\delta^{-t} B^{-n(A-c_1\delta-c_2\varepsilon)}.
\end{aligned}
\]
It follows that
\[
\begin{aligned}
  \sum_i|D_i^{(t)}|^s  \leq {}&c(\delta,\varepsilon)\,e^{s\rho^{(\varepsilon)}(n+i_t)}
  Z_{n+i_1-1}\big(s(1-\varepsilon)\big)\cdot B^{-n(A-c_1\delta-c_2\varepsilon)}.
\end{aligned}
\]

\subsection{Conclusion}

Summing over $t=1,\ldots,k$, and using the fact that $k$ is fixed,
we obtain
\[
\begin{aligned}
\limsup_{n\to\infty}\frac1n\log\sum_i|D_i|^s\leq P\big(s(1-\varepsilon)\big)-
\big(A-c_1\delta-c_2\varepsilon\big)\log B.
\end{aligned}\]
Here we used \(\frac{\rho^{(\varepsilon)}(n+i_t)}{n}  \longrightarrow0\)and
\[
\begin{aligned}
\frac1n
\log Z_{n+i_1-1}\big(s(1-\varepsilon)\big)
&=
\frac{n+i_1-1}{n}
P_{n+i_1-1}\big(s(1-\varepsilon)\big)\longrightarrow P\big(s(1-\varepsilon)\big),
\end{aligned}
\]
since $i_1$ is fixed.

Since $s>s_0(B)$, we have \( P(s)<A(s)\log B=A\log B.\)
By the continuity of $P$, we may first choose $\varepsilon>0$
sufficiently small, with \(s(1-\varepsilon)>\frac1d,\)
and then choose $\delta>0$ sufficiently small so that
\begin{equation}\label{eqn:finish}
  P\big(s(1-\varepsilon)\big)<\big(A-c_1\delta-c_2\varepsilon\big)\log B.
\end{equation}
It follows that there is some $\gamma<1$ such that, for all
sufficiently large $n$, \(\sum_i|D_i|^s\leq c\gamma^n.\)

Finally, every interval $D_i^{(t)}$ is contained in a cylinder whose
prefix has length at least $n+i_1-1$.    By the uniform
contraction property, the cylinder lengths tend to zero
uniformly as their level tends to infinity  Therefore, \[\max_i|D_i| \leq \max_{|\omega|=n+i_1-1}|C[\omega]| \longrightarrow0, \quad \text{ as } n \to \infty.\]
Thus \( \mathcal H^s(E(B;i,t))=0,\) and hence \( \dim_HE(B;i,t)\leq s.\)
Letting $s\searrow s_0(B)$ completes the proof of the upper bound.

\section{Proof of the Lower bound for $\dim_H E(B;i,t)$}\label{LBS}

To obtain the lower bound, we will construct an appropriate Cantor subset of set $E(B;i,t)$ and then apply the following mass distribution principle \cite{Fa_14}.
\begin{pro}[Mass Distribution Principle]
Let $\mu$ be a probability measure supported on a measurable set $F$. Suppose there are positive constants $c$ and $r_0$ such that
$$\mu\big(B(x,r)\big)\le c r^s$$
for any ball $B(x,r)$ with radius $r\le r_0$ and center $x\in F$. Then $\dim_{H} F\ge s$.\end{pro}

\subsection{Construction of a subset $F\subset E(B;i,t)$}

Fix \( \frac1d<s<s_0(B).\)
Then, by the definition of $s_0(B)$, \(P(s)>A(s)\log B.\) Choose \( b=(b_1,\ldots,b_k)\in S\)
such that \(A( b,s)=A(s).\) Since \(P_M(s)\longrightarrow P(s)\) as \(M\to\infty,\)
we may choose $M$ sufficiently large that
 \( P_M(s)>A(\mathbf b,s)\log B=A(s)\log B.\)
Let $\nu_s$ be the corresponding $s$-measure.

Fix small $\varepsilon,\delta>0$. Choose a sufficiently sparse
sequence of integers \( n_1<n_2<\cdots\)
such that the special blocks are disjoint and
\[
  \frac{n_1+\cdots+n_{j-1}}{n_j}\longrightarrow0.
\]

Let $F\subset [0,1)$ be the set of points whose symbolic expansions $a_i(x)$ satisfy, for every $x\in F$, the following conditions: 
\begin{itemize}
\item $a_i \in \{1,\ldots,M\}$ for $i\leq n_1+i_1-1$ and for $n_j+i_k+1\leq i \leq n_{j+1}+i_1-1$,
\item $a_i=1$ for $n_j+i_m+1\leq i \leq n_j+i_{m+1}-1$,
\item $\lfloor B^{n_j b_m}\rfloor+1 \leq a_i \leq 2\lfloor B^{n_j b_m}\rfloor$ for $i=n_j+i_m$.
\end{itemize}
Clearly, $F\subset E$.

\subsection{Construction of a measure $\mu$ supported on $F$}

To estimate from below $\dim_H E(B;i,t)$ it is enough to estimate $\dim_H F$, and we will do that using the Mass Distribution Principle. To apply it we will introduce the following notation: given a (finite or infinite) sequence $\omega$, we denote $\omega_\ell^m := \omega_\ell, \omega_{\ell+1},\ldots,\omega_m$. In this notation, let $\mu$ be as follows:
\begin{itemize}
\item for $n\leq n_1+i_1-1$ we have
\[
\mu(C[\omega_1^n]) = \nu_s(C[\omega_1^n]),
\]
\item for $n\in [n_j+i_k+1, n_{j+1}+i_1-1]$ we have
\[
\mu(C[\omega_1^n]) = \mu(C[\omega_1^{n_j+i_k}]) \cdot \nu_s(C[\omega_{n_j+i_k+1}^n]),
\]
\item for $n=n_j+i_k$ $\mu(C[\omega_1^n])$ is nonzero only if $\omega_{n_j+i_1}^{n_j+i_k} = \ell_1 1^{i_2-i_1-1} \ell_2 1^{i_3-i_2-1} \ldots 1^{i_k-i_{k-1}-1} \ell_k$, with $\lfloor B^{n_j b_m}\rfloor+1 \leq \ell_m \leq 2\lfloor B^{n_j b_m}\rfloor$, and then it is equal to
\[
\mu(C[\omega_1^n]) = \mu(C[\omega_1^{n_j+i_1-1}]) \cdot \prod_{m=1}^k \lfloor B^{n_j b_m}\rfloor^{-1}.
\]
\end{itemize}
That is, for symbols in the range $n_j+i_1,\ldots, n_j+i_k$ we uniformly distribute the measure according to the choice of the set $F$, while for symbols between positions $n_j+i_k+1$ and $n_{j+1}+i_1-1$, including the initial segment from 1 to $n_1+i_1-1$, we use the $s$-measure $\nu_s$.
The masses at intermediate levels are obtained by summing the
masses of their descendants. These assignments are consistent
and therefore define a probability measure supported on $F$.

\subsection{Calculating the lower local dimension of $\mu$}

The last step in the proof is to calculate the lower local dimension of measure $\mu$ at any point in $F$. Fix $x\in F$ and let $\omega$ be its symbolic expansion. We need to estimate from below
\begin{equation} \label{eqn:dmux}
\underline{d}\mu(x) := \liminf_{r\to 0} \frac {\log \mu(B_r(x))} {\log r}
\end{equation}
and to prove it has a lower bound, which can be made arbitrarily close to $s$ by taking larger $M$ and sparser $(n_j)_{j\geq 1}$.

To simplify the notation let $r_n=|C[(\omega_1^n)]|$ and note that we only need to check the right hand side of \eqref{eqn:dmux} for a subsequence of $r$'s with bounded multiplicative gaps. We will do the calculation by considering several cases, and note that the order of estimation is: Case 1, Case 2 for $j=1$, Case 3 for $j=1$, Case 2 for $j=2$, Case 3 for $j=2$, and so on.

Case 1: $r=r_n, n\leq n_1+i_1-1$.

In this case we have
\[
\mu(B_{r_n}(x)) \approx \mu(C[\omega_1^n]) \leq r_n^s \cdot e^{-nP_M(s)} \cdot e^{\rho(n)},
\]
and by Proposition \ref{CA} we have $P_M(s)\geq P(s)-\epsilon> 0$ for $M$ large enough and $\rho(n)/n\to 0$ (and hence can be made arbitrarily small around $n_1$ if $n_1$ is large), we have
\[
\mu(B_r(x)) \leq Kr^s
\]
for some uniform constant $K$.

Case 2: $r_{n_j+i_k}\leq r\leq r_{n_j+i_1-1}$.

From Case 1 for $j=1$, and Case 3 for $j\geq2$, the mass
$\mu(C[\omega_1^{n_j+i_1-1}])$ is bounded above by
$r_{n_j+i_1-1}^s e^{-(n_j+i_1-1)P(s)}$, up to the error factors
$e^{\rho(n_j+i_1-1)}$ and
$e^{(n_j+i_1-1)(P(s)-P_M(s))}$ and the errors from the preceding
blocks. By first choosing $M$ sufficiently large in terms of
$\varepsilon$ and then choosing $(n_j)$ sufficiently sparse,
these errors can be absorbed into
$r_{n_j+i_1-1}^{-s\varepsilon}$. Consequently,
\[
\mu(C[\omega_1^{n_j+i_1-1}])
\leq
r_{n_j+i_1-1}^{s(1-\varepsilon)}
e^{-(n_j+i_1-1)P(s)}.
\]

Moreover, the Almost Tempered Distortion gives
\[
\Dist(T_{\omega_{1}}\circ \ldots \circ T_{\omega_ {n_j+i_1-1} }) \leq r_{n_j+i_1-1}^{-\varepsilon} \cdot e^{\rho^{(\varepsilon)}(n_j+i_1-1)}
\]
and for large $n_j$ the right hand side can be estimated by $r_{n_j+i_1-1}^{-2\varepsilon}$.

For $2\leq m\leq k$ observe that the maps $  T_{\omega_{n_j+i_{m-1}+1}    \circ \ldots \circ   T_{\omega_{n_j+i_m-1}}  }=T_1^{i_m-i_{m-1}-1}$ are iterations of uniformly bounded number of copies of the same map $T_1$, hence they have uniformly bounded derivative and uniformly bounded distortion. We are thus going to ignore those ranges of $r$ and concentrate on the meaningful areas, with $r$ related to $r_{n_j+i_m}$. 

Let us start with $m=1$. If $b_1=0$, the first special position
contributes only a bounded factor, so there is nothing to prove at
this position. We may therefore assume that $b_1>0$.
 In this area the measure $\mu$ is distributed the following way: we have $\lfloor B^{n_j b_1}\rfloor$ cylinders, each of roughly the same size, and positioned next to each other. Thus, to get the minimal right hand side of \eqref{eqn:dmux}, we need to look at 
\[
r=\left|\bigcup_{w=\lfloor B^{n_j b_1}\rfloor+1}^{2\lfloor B^{n_j b_1}\rfloor} C[\omega_1^{n_j+i_1-1} w]\right|
\]
and we have
\[
\mu(B_r(x)) \leq K \mu(C[\omega_1^{n_j+i_1-1}]),\]
where $K\geq 1$ is a constant independent of $r,x,j.$

Moreover, for some constant $K(\delta)>0$ independent of $j,$
\[
r \geq r_{n_j+i_1-1} \cdot B^{n_j b_1} \cdot K^{-1}(\delta)B^{-(d+\delta) n_j b_1} \cdot r_{n_j+i_1-1}^{2\varepsilon}
\]
that is,
\[
r
\geq
K(\delta)^{-1}
r_{n_j+i_1-1}^{1+2\varepsilon}
B^{-(d-1+\delta)n_jb_1}.
\]

we can write
\[
\frac {\log \mu(B_r(x))} {\log r} \geq \frac {s(1-\varepsilon)\log r_{n_j+i_1-1} - (n_j+i_1-1)P(s)+\log K} {(1+2\varepsilon)\log r_{n_j+i_1-1}  -(d-1+\delta) n_j b_1 \log B-\log K(\delta)}. 
\]

Indeed, as $\log r_{n_j+i_1-1} < 0$, we have
\[
\frac {\log \mu(B_r(x))} {\log r} \geq \frac {s(1-\varepsilon) |\log r_{n_j+i_1-1}| + (n_j+i_1-1)P(s)-\log K} {(1+2\varepsilon) |\log{r_{n_j+i_1-1}}| + (d-1+\delta)n_j b_1 \log B+\log K(\delta)}.
\]

Put
\[
J_1=s(1-\varepsilon)|\log r_{n_j+i_1-1}|,
\qquad
J_2=(n_j+i_1-1)P(s),
\]
and
\[
J_3=(1+2\varepsilon)|\log r_{n_j+i_1-1}|,
\qquad
J_4=(d-1+\delta)n_jb_1\log B.
\]
Since $\log K$ and $\log K(\delta)$ are bounded, they contribute only
$O(n_j^{-1})$ to the quotient. Hence
\[
\frac{\log\mu(B_r(x))}{\log r}
\geq
\frac{J_1+J_2}{J_3+J_4}-O(n_j^{-1}).
\]
Using
\[
\frac{J_1+J_2}{J_3+J_4}
\geq
\min\left\{
\frac{J_1}{J_3},
\frac{J_2}{J_4}
\right\},
\]
we obtain
\[
\frac{J_1}{J_3}
=
s\frac{1-\varepsilon}{1+2\varepsilon}
\geq s-O(\varepsilon).
\]
Furthermore,
\[
P(s)>A(s)\log B
\geq A_1( b,s)\log B
=(d-1)sb_1\log B,
\]
and hence
\begin{align*}
\frac{J_2}{J_4}
&=
\frac{(n_j+i_1-1)P(s)}
{(d-1+\delta)n_jb_1\log B}
\\
&\geq
s\frac{n_j+i_1-1}{n_j}
\frac{d-1}{d-1+\delta}
\\
&\geq
s-O(\delta)-O(n_j^{-1}).
\end{align*}
Consequently,
\[
\frac{\log\mu(B_r(x))}{\log r}
\geq
s-O(\varepsilon)-O(\delta)-O(n_j^{-1}).
\]

The case of a general $m$ is similar. At the position $n_j+i_m$,
there are approximately $\lfloor B^{n_jb_m}\rfloor$ adjacent cylinders. Each has
size approximately
\[
r_{n_j+i_1-1}
B^{-dn_j(b_1+\cdots+b_m)}
\]
up to distortion, and measure approximately
\[
B^{-n_j(b_1+\cdots+b_m)}
\mu(C[\omega_1^{n_j+i_1-1}]).
\]
Taking $r$ to be the length of the union of all these cylinders,
and absorbing bounded factors and the fixed shift $i_1-1$ into the
error terms, we obtain
\[
\frac{\log\mu(B_r(x))}{\log r}
\geq
\frac{
s\log r_{n_j+i_1-1}
-n_j(b_1+\cdots+b_{m-1})\log B
-n_jP(s)+O(n_j\varepsilon+1)
}{
\log r_{n_j+i_1-1}
-n_j\left[
d(b_1+\cdots+b_{m-1})+(d-1)b_m
\right]\log B
+O(n_j\varepsilon+n_j\delta+1)
}.
\]
Equivalently,
\[
\frac{\log\mu(B_r(x))}{\log r}
\geq
\frac{
s|\log r_{n_j+i_1-1}|
+n_j(b_1+\cdots+b_{m-1})\log B
+n_jP(s)-O(n_j\varepsilon+1)
}{
|\log r_{n_j+i_1-1}|
+n_j\left[
d(b_1+\cdots+b_{m-1})+(d-1)b_m
\right]\log B
+O(n_j\varepsilon+n_j\delta+1)
}.
\]

For the main terms, put
\[
J_1
=
s|\log r_{n_j+i_1-1}|
+n_j(b_1+\cdots+b_{m-1})\log B,
\qquad
J_2=n_jP(s),
\]
and
\[
J_3
=
|\log r_{n_j+i_1-1}|
+n_js^{-1}(b_1+\cdots+b_{m-1})\log B,
\]
\[
J_4
=
n_js^{-1}A_m( b,s)\log B.
\]
The splitting of the denominator is valid because
\begin{align*}
&d(b_1+\cdots+b_{m-1})+(d-1)b_m
-s^{-1}(b_1+\cdots+b_{m-1})
\\
&\qquad
=
s^{-1}\left[
(ds-1)(b_1+\cdots+b_{m-1})+(d-1)sb_m
\right]
\\
&\qquad
=
s^{-1}A_m( b,s).
\end{align*}
Thus,
\[
\frac{\log\mu(B_r(x))}{\log r}
\geq
\frac{J_1+J_2}{J_3+J_4}
-O(\varepsilon)-O(\delta)-O(n_j^{-1}).
\]
We have
\[
\frac{J_1}{J_3}=s.
\]
If $A_m( b,s)>0$, then
\[
\frac{J_2}{J_4}
=
s\frac{P(s)}{A_m( b,s)\log B}
\geq s,
\]
because
\[
A_m( b,s)
\leq A( b,s)
=A(s)
\]
and
\[
P(s)>A(s)\log B.
\]
If $A_m( b,s)=0$, then
$b_1=\cdots=b_m=0$, and the corresponding positions contribute only
bounded factors. It follows that, in every case,
\[
\frac{\log\mu(B_r(x))}{\log r}
\geq
s-O(\varepsilon)-O(\delta)-O(n_j^{-1}).
\]
The case $m=1$ is obtained from this computation by interpreting the
empty sum $b_1+\cdots+b_{m-1}$ as zero.

Finally, if a ball meets only part of the consecutive admissible
cylinders at a special position, both its measure and its length are
proportional, up to the same distortion errors, to the number of
cylinders which it meets. Since $s<1$, the smallest local-dimension
estimate occurs when the whole collection of admissible cylinders is
taken. Therefore the preceding estimates hold for every
\[
r_{n_j+i_k}\leq r\leq r_{n_j+i_1-1}.
\]
This completes Case~2.

Case 3: $r_{n_{j+1}+i_1-1}\leq r{\leq r_{n_j+i_k}}$.

This case is similar to case 1. We have the initial cylinder $C[\omega_1^{n_j+i_k}]$ whose measure is, up to sub exponential factor, $$\nu_s(C[\omega_1^{n_1+i_1-1}]) \cdot \ldots \cdot \nu_s(C[\omega_{n_{j-1}+i_k+1}^{n_j+i_1-1}]) \cdot B^{-(n_1+\ldots +n_j)(b_1+\ldots+b_k)}$$, and we subdivide the measure inside it according to measure $\nu_s$. We remind from case 1 that
\[
\nu_s(B_r(y)) \leq K r^s
\]
for any $r$ and any $y$ (provided $M$ is large enough).

Moreover, by Almost Tempered Distortion,
\[
\Dist\left(
T^{n_j+i_k}\big|_{C[\omega_1^{n_j+i_k}]}
\right)
\leq r_{n_j+i_k}^{-2\varepsilon}
\]
for all sufficiently large $n_j$. Pulling the ball $B_r(x)$ back by
$T^{n_j+i_k}$ and applying the preceding estimate for $\nu_s$, we
obtain
\[
\mu(B_r(x))
\leq
K\mu(B_{r_{n_j+i_k}}(x))
r_{n_j+i_k}^{-2s\varepsilon}
\left(\frac{r}{r_{n_j+i_k}}\right)^s.
\]
Consequently,
\[
\frac{\log\mu(B_r(x))}{\log r}
\geq
\min\left\{
\frac{\log\mu(B_{r_{n_j+i_k}}(x))}
     {\log r_{n_j+i_k}}
-2s\varepsilon,
\,s
\right\}
-O(n_j^{-1}).
\]
By Case~2,
\[
\frac{\log\mu(B_{r_{n_j+i_k}}(x))}
     {\log r_{n_j+i_k}}
\geq
s-O(\varepsilon)-O(\delta)-O(n_j^{-1}).
\]
It follows that
\[
\frac{\log\mu(B_r(x))}{\log r}
\geq
s-O(\varepsilon)-O(\delta)-O(n_j^{-1})
\]
throughout the range
\[
r_{n_{j+1}+i_1-1}\leq r\leq r_{n_j+i_k}.
\]

Finally, since $\mathbf b$ is fixed in advance, we may choose
$n_{j+1}$ sufficiently large so that
\[
\mu(C[\omega_1^{n_{j+1}+i_1-1}])
\leq
r_{n_{j+1}+i_1-1}^{s(1-\varepsilon)}
e^{-(n_{j+1}+i_1-1)P(s)}.
\]
This is precisely the estimate required to apply Case~2 with
$j$ replaced by $j+1$. This completes Case~3.

Combining the three cases gives
\[
\underline d_\mu(x)\geq s-O(\varepsilon)-O(\delta)
\qquad\text{for every }x\in F.
\]
By the local-dimension form of the mass distribution principle,
\[
\dim_HE\geq\dim_HF\geq s-O(\varepsilon)-O(\delta).
\]
Making the auxiliary errors arbitrarily small and then letting
$s\uparrow s_0(B)$ proves the lower bound.

\section{Poof of Theorem \ref{MT2}}\label{genF}

In this section we prove Theorem \ref{MT2}. We consider the three case: $1<B<\infty, B=1, B=\infty.$

\medskip
\noindent
\subsection{\textbf{Case $1<B<\infty$.}}
Upper bound: 
Fix $\varepsilon>0,$  on has 
\[
\Phi(n)\geq(B-\varepsilon)^n
\]
for all sufficiently large $n.$ Hence
\[
E( \Phi; i,t       )\subset E(B-\varepsilon;   i,t      ),
\quad \quad
\dim_HE(  \Phi; i, t      )\leq s_0(B-\varepsilon).
\]
by Theorem  \ref{MT}. Recall Proposition \ref{dimB}   for the continuity of $s_{0}(B)$ with respect to $B.$ Then by letting $\epsilon \to 0,$ the upper bound for $\dim_HE(  \Phi; i, t      )$ follows.

Lower bound: The argument of the lower bound of $E(      \Phi; i, t    )$ is almost the same as for $E(B; i, t    )$ given in Section \ref{LBS}. So for simplicity we only give outline of proof.

 By the definition of the lower
limit, there are infinitely many $n\in\N$ such that
\[
\phi(n)\leq(B+\varepsilon)^n.
\]
Choose a sufficiently sparse sequence $n_1<n_2<\cdots$
among these integers. This is possible because there are infinitely
many such integers, and in the lower-bound construction each $n_j$
is only required to be sufficiently large compared with the preceding
ones.

Fix
$
\frac1d<s<s_0(B+\varepsilon).$ Choose $ b=(b_1,\ldots,b_k)\in S$ such that
$ A( b,s)=A(s).$ Since $P(s)>A(s)\log(B+\varepsilon),$
we may choose $M$ sufficiently large that
\[
P_M(s)>A(s)\log(B+\varepsilon).
\]

We now repeat the Cantor construction used in Section \ref{LBS} with $B$ replaced by $B+\varepsilon$ and the chosen sequence $(n_j)$. Namely, the digits in the free blocks belong
to $\{1,\ldots,M\}$, the digits between two consecutive special
positions are equal to $1$, and at the special positions we impose
\[
(B+\varepsilon)^{n_jb_m}+1
<
a_{n_j+i_m}(x)
\leq
2(B+\varepsilon)^{n_jb_m},
\qquad 1\leq m\leq k.
\]
Let $F$ be the resulting Cantor set.
For every $x\in F$ and every $j$,
\[
\begin{aligned}
\prod_{m=1}^{k}a_{n_j+i_m}(x)^{t_m}
&>
\prod_{m=1}^{k}
(B+\varepsilon)^{n_jb_mt_m} =
(B+\varepsilon)^{
n_j\sum_{m=1}^{k}b_mt_m} 
=
(B+\varepsilon)^{n_j}  \geq\Phi(n_j)
\end{aligned}
\]
 Hence \(F\subset E(\Phi;i,t).\) The measure $\mu$ on $F$ is defined exactly as in the 
lower-bound proof: on the free blocks the mass is distributed
according to $\nu_s$, and at every special position it is distributed
uniformly among the admissible digits. The estimates in Cases~1--3
of that proof remain unchanged, with $B$ replaced by
$B+\varepsilon$. 
It gives \(\dim_HF\geq s.\)
Since $F\subset E(\Phi;   i,t      )$,
letting $s\nearrow s_0(B+\varepsilon)$ gives
\[
\dim_HE(\Phi;    i, t       )\geq s_0(B+\varepsilon).
\]

Finally, letting $\varepsilon\searrow0$ and using the continuity of
$s_0(B)$ gives
$
{\dim_HE(\Phi;     i, t    )=s_0(B).}
$

\medskip
\noindent
\subsection{\textbf{Case $B=1$.}}

Consider 
\[
F_1(\Phi)
=
\left\{
x\in[0,1):
a_{n+i_1}(x)^{t_1}\geq\Phi(n)
\text{ for infinitely many }n\in\N
\right\}.
\]
We first prove that
\(
\dim_HF_1(\Phi)=1
\). Fix $\varepsilon>0$. Since
$
B=\liminf_{n\to\infty}\Phi(n)^{1/n}=1,$
there are infinitely many $n\in\N$ such that
\[
\Phi(n)\leq(1+\varepsilon)^n.
\]

Choose the sufficiently sparse sequence $(n_j)_{j\geq1}$ in the
one-digit lower-bound construction of Section \ref{LBS} among these integers and set
\[
(1+\varepsilon)^{n_j/t_1}+1
<
a_{n_j+i_1}(x)
\leq
2(1+\varepsilon)^{n_j/t_1}.
\]
Therefore, every point $x$ in the resulting Cantor set satisfies
\[
a_{n_j+i_1}(x)^{t_1}
>
(1+\varepsilon)^{n_j}
\geq
\Phi(n_j)
\]
for every $j$. Hence this Cantor set is contained in $F_1(\Phi)$.
For all sufficiently small $\varepsilon,$ the one digit critical exponent is greater than $1/d,$ by Proposition \ref{dimB}. The one digit lower bound calculations, with $b_{1}=1/t_{1},$ therefore gives
\[
\dim_HF_1(\Phi)
\geq
\inf\left\{
s>\frac1d:
P(s)\leq
\frac{(d-1)s}{t_1}\log(1+\varepsilon)
\right\}.
\]

By Proposition~\ref{dimB}, the right hand side tends to $1$ as
$\varepsilon\searrow0$ (because $P(1)=0$ and $P(s)>0$ for every $s<1$). Consequently,
\[
\dim_HF_1(\Phi)=1.
\]
 Finally, since every digit is at least one and every weight is
positive,
\[
\prod_{m=1}^{k}a_{n+i_m}(x)^{t_m}
\geq
a_{n+i_1}(x)^{t_1}.
\]
It follows that
\[
F_1(\Phi )\subset E(\Phi; i, t).
\]
Therefore,
\[
1
=
\dim_HF_1(\Phi)
\leq
\dim_HE(\Phi; i, t)
\leq1.
\]

\subsection{\textbf{Case $B=\infty$.}}
We distinguish the three possible values of $b$.

\smallskip
\noindent
\subsubsection{\textbf{\emph{First suppose that $1<b<\infty$.}}}

Lower bound:  By the definition of $b$, for any $c>b$ there are infinitely many
$n$ such that
\(
\Phi(n)\leq e^{c^n}.
\)
Hence,
\[
\left\{
x:
a_{n+i_1}(x)
\geq
e^{\frac{c^n}{t_1}}
\text{ for all sufficiently large }n\in\N
\right\}
\subset E( \Phi; i, t       ).
\]
 Lemma \ref{luc} gives
\[
\dim_HE(    \Phi; i, t    )\geq\frac{1}{1+(d-1)c}.
\]
By the arbitrariness of $c>b,$ we obtain
\[
\dim_HE(\Phi; i, t )\geq  \frac{1}{1+(d-1)b}    .
\]

For the upper bound, let $1<c<b$. By the definition of $b$,
\(
\Phi(n)\geq e^{c^n}
\) for all sufficiently large $n$.

If
\[
\prod_{m=1}^{k}a_{n+i_m}(x)^{t_m}\geq\Phi(n),
\]
then, for at least one $m\in\{1,\ldots,k\}$,
\[
a_{n+i_m}(x)^{t_m}\geq\Phi(n)^{1/k}.
\]
Consequently,
\[
E(  \Phi; i, t      )
\subset
\bigcup_{m=1}^{k}
\left\{
x:
a_{n+i_m}(x)
\geq
e^{\frac{c^n}{kt_m}}
\text{ for infinitely many }n
\right\}.
\]
It follows from Lemma \ref{luc} that each of the sets on the right hand side have Hausdorff dimension $\frac{1}{1+(d-1)c}$ irrespective of $t'_{m}s.$ Hence 
$\dim_HE(  \Phi; i, t    )\leq\frac{1}{1+(d-1) b},$ by the arbitrariness of $c<b.$

\smallskip
\noindent
\subsubsection{\textbf{\emph{Now suppose that $b=\infty$.}}}

Fix any $c>1$. By the definition of $b$,
\(
  \Phi(n)\geq e^{c^n}
\)
for all sufficiently large $n$. As above,
\[
E(\Phi;i,t)
\subset
\bigcup_{m=1}^{k}
\left\{
x:
a_{n+i_m}(x)
\geq
e^{\frac{c^n}{kt_m}}
\text{ for infinitely many }n\in\N
\right\}.
\]
Lemma~\ref{luc} therefore gives
\[
  \dim_H E(  \Phi; i, t    )
  \leq
  \frac{1}{1+(d-1)c}.
\]
Letting $c\to\infty$, we obtain
\[
  \dim_H E(  \Phi; i, t      )=0.
\]

\smallskip
\noindent
\subsubsection{\textbf{\emph{Finally suppose that $b=1$.}}}

Lower bound: 

In this case for any $\epsilon>0,$ $\Phi(n)\leq e^{(1+\epsilon)^n}$ for infinitely many $n.$
Therefore,
\[
\left\{
x:
a_{n+i_1}(x)
\geq
e^{\frac{(1+\epsilon)^n}{t_1}}
\text{ for all sufficiently large }n
\right\}
\subset E(  \Phi; i, t    ).
\]
Hence, Lemma \ref{luc} gives
\[
\dim_HE(  \Phi; i, t    )\geq \lim_{\epsilon\to 0}\frac{1}{1+(d-1)(1+\epsilon)}=\frac{1}{d}.
\]

For the upper bound, since $B=\infty$, for every $C>1$ we have
\[
  \Phi(n)\geq C^n
\]
for all sufficiently large $n$. Hence
\[
  E( \Phi; i, t       )\subset E(C;i,t).
\]
By Theorem \ref{MT},
\[
  \dim_H E(  \Phi; i, t   )\leq  \dim_H E(C;i,t)  \leq s_0(C).
\]
Letting $C\to\infty$ and using
Proposition~\ref{dimB}, we obtain
\[
  \dim_H E(\Phi; i, t      )\leq\frac1d.
\]

This completes the proof of Theorem \ref{MT2}.

\qed

\providecommand{\bysame}{\leavevmode\hbox to3em{\hrulefill}\thinspace}
\providecommand{\MR}{\relax\ifhmode\unskip\space\fi MR }
\providecommand{\MRhref}[2]{%
  \href{http://www.ams.org/mathscinet-getitem?mr=#1}{#2}
}
\providecommand{\href}[2]{#2}

\end{document}